\documentclass[11pt]{amsart}

\usepackage[utf8]{inputenc}

\usepackage{mathabx}
\usepackage{amsmath}
\usepackage{amsthm}
\usepackage{amssymb}
\usepackage{amscd}
\usepackage{amsfonts}
\usepackage{amsbsy}
\usepackage{graphicx}
\usepackage[dvips]{psfrag}
\usepackage{array}
\usepackage{color}
\usepackage{epsfig}
\usepackage{url}
\usepackage{enumerate}

\usepackage{overpic}
\usepackage{epstopdf}
\usepackage{hyperref}

\newcommand{\R}{\ensuremath{\mathbb{R}}}

\newcommand{\D}{\ensuremath{\mathcal{D}}}

\newcommand{\x}{\mathbf{x}}

\newcommand{\sgn}{\mathrm{sign}}

\newcommand{\Int}{\mathrm{Int}}

\newcommand{\bb}{\mathbf{b}}

\newtheorem {theorem} {Theorem}

\newtheorem {proposition} {Proposition}

\newtheorem {lemma}  {Lemma}

\newtheorem {remark} {Remark}

\usepackage{mathpazo}

\begin{document}
\renewcommand{\arraystretch}{1.5}

\title[Characterization of period annuli for two-zonal planar PWLS]
{A succinct characterization of period annuli\\  in planar piecewise linear differential systems\\ with a straight line of nonsmoothness}

\author[V. Carmona, Fernandez-S\'{a}nchez, and D. D. Novaes]
{Victoriano Carmona$^1$, Fernando Fern\'{a}ndez-S\'{a}nchez$^2$,\\ and Douglas D. Novaes$^3$}

\address{$^1$ Dpto. Matem\'{a}tica Aplicada II \& IMUS, Universidad de Sevilla, Escuela Polit\'ecnica Superior.
Calle Virgen de \'Africa 7, 41011 Sevilla, Spain.} 
 \email{vcarmona@us.es} 

\address{$^2$ Dpto. Matem\'{a}tica Aplicada II \& IMUS, Universidad de Sevilla, Escuela T\'{e}cnica Superior de Ingenier\'{i}a.
Camino de los Descubrimientos s/n, 41092 Sevilla, Spain.} \email{fefesan@us.es}

\address{$^3$ Departamento de Matem\'{a}tica, Instituto de Matem\'{a}tica, Estatística e Computa\c{c}\~{a}o Cient\'{i}fica (IMECC), Universidade
Estadual de Campinas (UNICAMP), Rua S\'{e}rgio Buarque de Holanda, 651, Cidade Universit\'{a}ria Zeferino Vaz, 13083--859, Campinas, SP,
Brazil.} \email{ddnovaes@unicamp.br} 

\subjclass[2010]{34A36, 34C25, 37C27}

\keywords{Piecewise linear systems, period annuli, Poincar\'e half-maps, integral characterization}

\begin{abstract}
We close the problem of the existence of period annuli in planar piecewise linear differential systems with a straight line of nonsmoothness. In fact, a characterization for the existence of such objects is provided by means of a few basic operations on the parameters.
\end{abstract}

\maketitle

\section{Introduction}
Determining sufficient and necessary conditions for the existence of a period annulus in planar differential systems is a classical problem in qualitative theory of planar vector fields. For the particular case that the period annulus ends in a monodromic singularity, such a problem is known as {\it Center Problem}, which was exhaustively studied for polynomial vector fields (see, for instance, \cite{romanovski2009center}). This problem has also been considered in the context of planar piecewise smooth differential systems (see, for instance, \cite{gassulcoll,coll1999,GasTor03,Novaes2022,Novaes2021,pleshkan73}).

However, due to the complexity imposed by the nonsmoothness, the center problem is not solved even for the simplest family of piecewise smooth differential systems, namely piecewise linear differential systems with two zones separated by the straight line $\Sigma=\left\{(x,y)\in\mathbb{R}^2:x=0\right\}$,
\begin{equation}\label{s1}
\dot \x =
\left\{\begin{array}{l}
A_L\x+\bb_L, \quad\textrm{if}\quad x\leq 0,\\
A_R\x+\bb_R, \quad\textrm{if}\quad x\geq 0.
\end{array}\right.
\end{equation}
Here, $\x=(x,y)\in\R^2,$ $A_{L}=(a_{ij}^{L})_{2\times 2},$ $A_{R}=(a_{ij}^{R})_{2\times 2},$ $\bb_{L}=(b_1^{L},b_2^{L})\in\R^2,$ $\bb_{R}=(b_1^{R},b_2^{R})\in\R^2$, and the dot denotes the derivative with respect to the independent variable $t$. The Filippov's convention \cite{Filippov88} is assumed for trajectories of \eqref{s1}.

The main goal of this paper is to close the problem of the existence of crossing period annuli for system \eqref{s1} by providing a characterization for the existence of such objects by means of a few basic operations on the parameters.

Since system \eqref{s1} is piecewise linear, two obvious conditions implying the existence of period annuli are: 
\begin{itemize}
\item[(A)] $T_L=0,$ $D_L>0,$ and $a_L<0$; or
\item[(B)] $T_R=0,$ $D_R>0,$ and $a_R>0$, 
\end{itemize}
where $T_L,$ $T_R$ and $D_L,$ $D_R$ denote, respectively, the traces and determinants of the matrices $A_L$ and $A_R$ and
\begin{equation}
\label{ecu:aLaR}
a_{L}=a_{12}^{L}b_2^{L}-a_{22}^{L}b_1^{L}\quad \mbox{and} \quad  a_{R}=a_{12}^{R}b_2^{R}-a_{22}^{R}b_1^{R}.
\end{equation}
Indeed, condition (A) implies that system \eqref{s1} has a linear center (and so a period annulus) in the half-plane $\left\{ (x,y)\in\mathbb{R}:x< 0\right\}$, and condition (B) implies that system \eqref{s1} has a linear center (and so a period annulus) in the half-plane $\left\{ (x,y)\in\mathbb{R}:x> 0\right\}$. 

Apart the trivial cases above, system \eqref{s1} admits period annuli whose orbits cross the separation line $\Sigma$. Regarding those period annuli we may quote the following papers. In \cite{FreireEtAl12}, Freire et al. provided sufficient conditions for piecewise linear systems of kind \eqref{s1} formed by two foci and without sliding set to have a global center around the origin. In \cite{BuzziEtAl13}, Buzzi et al. classified the centers at infinity for piecewise linear perturbations of linear centers. In \cite{MedradoTorregrosa15}, Medrado \& Torregrosa established sufficient conditions in order for a monodromic singularity at the separation line $\Sigma$ to be a center. Finally, in \cite{FREIRE2021124818}, Freire et al.~ characterize when systems of kind \eqref{s1}, formed by two foci, have a center at infinity.

In this paper, we present a general and concise characterization of the existence of a crossing period annulus for system \eqref{s1}. This characterization will be given in terms of their parameters and, unlike the mentioned papers, regardless the local nature of each linear system.

Notice that the existence of a crossing periodic orbit for system \eqref{s1} implies trivially the existence of the Poincar\'{e} half-maps associated with $\Sigma$.
In turn, such maps exist if, and only if, the following set of conditions hold:
\[
\text{(H)}:\left\{\begin{array}{l}
 a^L_{12}a^R_{12}>0;\\
 a_L\leq0 \,\,\text{and}\,\, 4D_L-T_L^2>0, \,\,\text{or}\,\, a_L>0;\\
 a_R\geq0 \,\,\text{and}\,\, 4D_R-T_R^2>0, \,\,\text{or}\,\, a_R<0.
 \end{array}\right.
\]
Indeed, taking into account the direction of the flow on the separation line $x=0$, it is straightforward to see that the inequality $a_{12}^{L}a_{12}^{R}>0 $ is necessary for the existence of crossing periodic solutions. The other two conditions will be discussed below (see Propositions \ref{prop:defyL} and \ref{prop:defyR}).

Now, we present the main result of this paper.
\begin{theorem}\label{main} 
Consider the planar piecewise linear differential system \eqref{s1}. 
Let $T_L,$ $T_R$ and $D_L,$ $D_R$ be, respectively, the traces and determinants of the matrices $A_L$ and $A_R$ and let $a_L$ and $a_R$ be the values given in expression \eqref{ecu:aLaR}.
Denote 
\begin{equation}
\label{ec:losxiybeta}
\xi_0:=a_RT_L-a_LT_R,\,\,\,\xi_{\infty}:=T_L^2D_R-T_R^2D_L,\quad \mbox{and}\quad \beta:=a_{12}^Lb_1^R-b_1^La_{12}^R.
\end{equation}
Then, the differential system \eqref{s1} has a crossing period annulus 
if, and only if, the condition (H) holds, $\sgn(T_R)=-\sgn(T_L)$, and  $\xi_0=\xi_{\infty}=\beta=0$.
\end{theorem}

At this point, we must clarify the dynamical meanings of the values $\xi_0$, $\xi_{\infty}$, and $\beta$ and of the relationship $\sgn(T_R)=-\sgn(T_L)$. 

First, under the hypothesis $a^L_{12}a^R_{12}>0$, system \eqref{s1} has a sliding region contained in $\Sigma$ and delimited by the points $\left(0,-b_1^L/a_{12}^L\right)$ and $\left(0,-b_1^R/a_{12}^R\right)$ provided that $\beta$ does not vanish. Accordingly, the condition $\beta=0$ indicates that system \eqref{s1} does not have any sliding region.

Second, when system \eqref{s1} does not have a sliding region, as it follows from Proposition 14 of \cite{Caretalpre22}, the sign of the value $\xi_0$ (called by $\xi$ in that work) provides the stability of the origin of system \eqref{s1} when it is a monodromic singularity. Moreover, from Proposition 15 of \cite{Caretalpre22}, under the assumption $\sgn(T_R)=-\sgn(T_L)\ne 0$, the sign of the value $c_\infty=T_L\xi_\infty$ determines the stability of the infinity for system \eqref{s1} when it is monodromic.

Finally, since system \eqref{s1} is linear on each side of the separation straight line $\Sigma$, the signs of the traces $T_L$ and $T_R$ determine the (area) contraction/expansion of the system on each side of $\Sigma$ and so the condition  $\sgn(T_R)=-\sgn(T_L)$  ensures a kind of balance between the contraction of the system in one zone and  the expansion of the system in the other zone. 

Theorem \ref{main} is proven in Section \ref{sec:proof}. Its proof is based on a recent integral characterization for Poincar\'{e} half-maps for planar linear differential systems introduced in \cite{CarmonaEtAl19} by Carmona \& Fern\'andez-S\'{a}nchez, which has been successfully used to analyze periodic behavior of piecewise linear systems (see, for instance, \cite{CARMONA2021100992,Caretalpre22,Carmona2022}). This characterization as well as some useful properties of the Poincar\'{e} half-maps will be introduced in Section \ref{sec:prel}.

\section{Poincar\'{e} half-maps and displacement function: some preliminary results}\label{sec:prel}
In this section, after introducing a canonical form for system \eqref{s1} in Subsection \ref{sec:cf}, the definition of the Poincar\'{e} half-maps for planar linear differential systems will be presented in Subsection \ref{sec:icphm}. Some useful properties of these maps, provided in \cite{Caretalpre21}, will be collected in Subsection \ref{sec:prop}. In Subsection \ref{sec:disp}, a displacement function will be given together with some of its main features.

\subsection{Canonical form} \label{sec:cf}
As it was said in Introduction, the existence of crossing periodic solutions of system \eqref{s1} implies straightforwardly the first condition of Hypothesis (H), that is, $a_{12}^{L}a_{12}^{R}>0$. Moreover, under this condition, Freire et al. in \cite{FreireEtAl12} stated that the differential system \eqref{s1} is reduced, by a homeomorphism preserving the separation line $\Sigma=\left\{(x,y)\in\R^2:\,x=0\right\}$, into the following Li\'enard canonical form 
\begin{equation}\label{cf}
\left\{\begin{array}{l}
\dot x= T_L x-y,\\
\dot y= D_L x-a_L,
\end{array}\right.\quad \text{for}\quad x\leqslant0, 
\quad 
\left\{\begin{array}{l}
\dot x= T_R x-y+b,\\
\dot y= D_R x-a_R,
\end{array}\right.\quad \text{for}\quad x\geqslant0,
\end{equation}
being $a_L$ and $a_R$ the values given in expression \eqref{ecu:aLaR},  $T_L,$ $T_R$ and $D_L,$ $D_R$, respectively, the traces and determinants of the matrices $A_L$ and $A_R$, and $b=\beta/a_{12}^R$, where $\beta$ is given in expression \eqref{ec:losxiybeta}.

\subsection{Integral characterization of Poincar\'{e} half-maps}  \label{sec:icphm}

The periodic solutions of the piecewise linear differential system \eqref{cf} are studied via two Poincar\'{e} Half-Maps defined on $\Sigma$: the {\it Forward Poincar\'{e} Half-Map}  $y_L: I_L\subset [0,+\infty) \longrightarrow(-\infty,0]$ and  the {\it Backward Poincar\'{e} Half-Map} $y_R^b:I_R^b\subset [b,+\infty)\rightarrow (-\infty,b]$.  

On the one hand, the forward Poincar\'{e} half-map takes a point $(0,y_0)$, with $y_0\geq0$, and maps it to a point $(0,y_L(y_0))$ by traveling through the flow of \eqref{cf} in the positive time direction. Clearly, it is determined by the left linear differential system of \eqref{cf} and its formal definition will be given in Proposition \ref{prop:defyL}. 

On the other hand, the backward Poincar\'{e} half-map takes a point $(0,y_0)$, with $y_0\geq b$, and maps it to a point $(0,y_R^ b(y_0))$ by  traveling through the flow of \eqref{cf} in the negative time direction. Clearly, it is determined by the right linear differential system of \eqref{cf}. Notice that the simple translation $y \mapsto y-b$ applied to this right linear system allows us to write $y_R^b(y_0)=y_R(y_0-b)+b$ and $I_R^b=I_R+b$, where
$y_R:I_R\subset [0,+\infty)\rightarrow (-\infty,0]$ is the backward Poincar\'{e} half-map of \eqref{cf} for $b=0$, that is,  $y_R=y_R^0$ and $I_R=I_R^0$. The formal definition of the map $y_R$ and its domain $I_R$ will be given in Proposition \ref{prop:defyR}.

In Propositions \ref{prop:defyL} and \ref{prop:defyR}, we will need the following concept of {\it Cauchy Principal Value}:
 \[
\operatorname{PV}\left\{\int_{y_1}^{y_0}f(y)dy\right\}:=\lim_{\varepsilon\searrow 0} \left(\int_{y_1}^{-\varepsilon}f(y)dy+\int_{\varepsilon}^{y_0}f(y)dy\right),
\]
 for $y_1<0<y_0$ and $f$  continuous in $[y_1,y_0]\setminus \{0\}$ (see, for instance, \cite{henrici}). Note that if $f$ is also continuous at $0$, then the Cauchy principal value coincides with the definite integral.

The forward Poincar\'{e} half-map $y_L$ refers to the linear system
\begin{equation}
\label{ecu:linealL}
\left\{\begin{array}{l}
\dot x= T_L x-y,\\
\dot y= D_L x-a_L,
\end{array}\right.
\end{equation}
which corresponds with the left linear system of \eqref{cf}. Thus, its definition, its domain $I_L$, and its analyticity are given by Theorem 19, Corollary 21, and Corollary 24 of \cite{CarmonaEtAl19}.
In the following proposition, we summarize the mentioned results (see  \cite[ Theorem 1]{Caretalpre22}).
\begin{proposition}
\label{prop:defyL}
The forward Poincar\'{e} half-map $y_L$ is well defined if, and only if, $a_L\leqslant 0$ and $4D_L-T_L^2>0$, or $a_L>0$. In this case, $I_L:= [\lambda_L,\mu_L)\ne\emptyset$ and the following statements hold:
\begin{enumerate}[(a)]
\item   \label{itemayL} The right endpoint $\mu_L$ of the interval $I_L$ is the smallest strictly positive root of the polynomial $W_L(y)=D_Ly^2-a_LT_Ly+a_L^2$, 
if it exists. Otherwise, $\mu_L=+\infty$.
\item \label{itembyL}The left endpoint $\lambda_L$ of the interval $I_L$ is greater than or equal to zero. If  $\lambda_L>0$, then $y_L(\lambda_L)=0$, $a_L<0$, $4D_L-T_L^2>0$, and $T_L<0$. 
Moreover, if $y_L(\lambda_L)<0$, then $\lambda_L=0$ and $a_L<0$, $4D_L-T_L^2>0$, and $T_L>0$. 
\item\label{itemcYL} The polynomial $W_L$ verifies $W_L(y)>0$ for $y \in \operatorname{ch}(I_L\cup y_L(I_L))\setminus\{0\}$, where $\operatorname{ch} (\cdot)$ denotes the convex hull of a set.
\item The forward Poincar\'{e} half-map $y_L$ is the unique function $y_L: I_L\subset [0,+\infty) \longrightarrow(-\infty,0]$
that satisfies 
\begin{equation}\label{integralF}
\operatorname{PV}\left\{\int_{y_L(y_0)}^{y_0}\dfrac{-y}{W_L(y)}dy\right\}= q_L(a_L,T_L,D_L)
\quad \mbox{for} \quad y_0\in I_L, 
\end{equation}
where
\begin{equation}
\label{ecu:qL}
q_L(a_L,T_L,D_L)=\left\{
\begin{array}{ccl}
0 & \mathrm{if} & a_L>0, \\
\frac{\pi T_L}{D_L\sqrt{4D_L-T_L^2}}
& \mathrm{if} & a_L=0, \\
\frac{2\pi T_L}{D_L\sqrt{4D_L-T_L^2}}
& \mathrm{if} & a_L<0.
\end{array}
\right.
\end{equation}
\item The forward Poincar\'e half-map $y_L$ is analytic in $\Int(I_L)$.
\end{enumerate}
 \end{proposition}

On the other hand, the backward Poincar\'{e} half-map $y_R$ refers to the linear system
\begin{equation*}
\label{ecu:linealR}
\left\{\begin{array}{l}
\dot x= T_R x-y,\\
\dot y= D_R x-a_R,
\end{array}\right.
\end{equation*}
which corresponds with the right linear system of \eqref{cf} for $b=0$.
Thus, its definition, its domain $I_R$, and its analyticity are obtained from Proposition \ref{prop:defyL} by means of the change of variables $(t,x)\mapsto(-t,-x)$ and taking $(a_L,D_L,T_L)=(-a_R,D_R,-T_R)$ in system \eqref{ecu:linealL} (see  \cite[ Theorem 2]{Caretalpre22}).
 \begin{proposition}
\label{prop:defyR}
The backward Poincar\'{e} half-map $y_R$ is well defined if, and only if, $a_R\geqslant 0$ and $4D_R-T_R^2>0$, or $a_R<0$. In this case, $I_R:=[\lambda_R,\mu_R)\ne\emptyset$ and the following statements hold:
\begin{enumerate}[(a)]
\item The right endpoint $\mu_R$ of its definition interval $I_R$ is the smallest strictly positive root of the polynomial $W_R(y)=D_Ry^2-a_RT_Ry+a_R^2$, 
if it exists. Otherwise, $\mu_R=+\infty$.
\item \label{itembdeyR}
The left endpoint $\lambda_R$ of the interval $I_R$ is greater than or equal to zero.   If  $\lambda_R>0$, then $y_R(\lambda_R)=0$, $a_R>0$, $4D_R-T_R^2>0$, and $T_R>0$.  Moreover, if $y_R(\lambda_R)<0$, then $\lambda_R=0$ and $a_R>0$, $4D_R-T_R^2>0$, and $T_R<0$. 
\item The polynomial $W_R$ verifies $W_R(y)>0$ for $y \in \operatorname{ch}(I_R\cup y_R(I_R))\setminus\{0\}$.
\item 
The backward Poincar\'{e} half-map $y_R$ is the unique function $y_R: I_R\subset [0,+\infty) \longrightarrow(-\infty,0]$
that satisfies 
\begin{equation}\label{integralB}
\operatorname{PV}\left\{\int_{y_R(y_0)}^{y_0}\dfrac{-y}{W_R(y)}dy\right\}= q_R(a_R,T_R,D_R)
\quad \mbox{for} \quad y_0\in I_R,
\end{equation}
where
\begin{equation}
\label{ecu:qR}
q_R(a_R,T_R,D_R)=\left\{
\begin{array}{ccl}
0 & \mathrm{if} & a_R<0, \\
-\frac{\pi T_R}{D_R\sqrt{4D_R-T_R^2}}
& \mathrm{if} & a_R=0, \\
-\frac{2\pi T_R}{D_R\sqrt{4D_R-T_R^2}}
& \mathrm{if} & a_R>0.
\end{array}
\right.
\end{equation}
\item The backward Poincar\'e half-map $y_R$ is analytic in $\Int(I_R)$.
\end{enumerate}
\end{proposition}

\begin{remark}
\label{rm:yRyLaLaRzero}
Notice that the integral  that appears in \eqref{integralF} (resp. \eqref{integralB}) is divergent for $a_L=0$ (resp. $a_R=0$). Nevertheless, in this case, the Cauchy principal value provides
\begin{equation}
\label{ecu:yLyforaLaRnull}
y_L(y_0)=-e^{\frac{\pi T_L}{\sqrt{4D_L-T_L^2}}}y_0,\quad\left( \text{resp.}\quad y_R(y_0)=-e^{\frac{-\pi T_R}{\sqrt{4D_R-T_R^2}}}y_0\right) \quad y_0\geq 0.
\end{equation}
In any other case, that is, $a_L\ne 0$ (resp. $a_R\ne0$), the Cauchy principal value can be removed because the integral is a proper integral.
\end{remark}

\subsection{Properties of Poincar\'{e} half-maps}\label{sec:prop}
Some useful properties of the Poincar\'{e} half-maps $y_L$  and $y_R$ will be collected in the next results. The proofs of these properties for the map $y_L$ are given in \cite{Caretalpre21} and they can be extended to $y_R$ by means of the change of variables $(t,x)\mapsto(-t,-x)$ and taking $(a_L,D_L,T_L)=(-a_R,D_R,-T_R)$ in system \eqref{ecu:linealL}. The first one (Proposition \ref{prop:derivadasprimeraysegundayLyR}) provides, as a direct consequence of expressions \eqref{integralF} and \eqref{integralB}, the first derivative of the Poincar\'{e} half-maps. The second result (Proposition \ref{rm:signoy0+y1}) establishes the relative position between the graph of the Poincar\'e half-maps and the bisector of the fourth quadrant. The third result (Proposition \ref{prop:serietaylor}) gives the first coefficients of the Taylor expansions of the Poincar\'e half-map $y_R$ at the origin. The last result (Proposition \ref{prop:serieNewton-Puiseux}) shows the first coefficient of the Newton-Puiseux series expansion of $y_L$ around a point $\widehat y_0>0$ such that $y_L(\widehat y_0)=0$. 
\begin{proposition}
\label{prop:derivadasprimeraysegundayLyR}
The first derivatives of the Poincar\'{e} half-maps $y_L$ and $y_R$ are given by 
\begin{equation*}
\label{eq:derivadaprimerayL}
y_L'(y_0)=\frac{y_0W_L(y_L(y_0))}{y_L(y_0)W_L(y_0)} <0   \quad \mbox{for} \quad y_0\in \operatorname{int}(I_L),
\end{equation*}
\begin{equation*}
\label{eq:derivadaprimerayR}
y_R'(y_0)=\frac{y_0W_R(y_R(y_0))}{y_R(y_0)W_R(y_0)}<0 \quad \mbox{for} \quad y_0\in \operatorname{int}(I_R),
\end{equation*}
where the polynomials $W_L$ and $W_R$ are given in Propositions \ref{prop:defyL} and \ref{prop:defyR}, respectively. 
\end{proposition}
\begin{proposition}
\label{rm:signoy0+y1}
The following statements hold.
\begin{enumerate}[(a)]
\item The forward Poincar\'{e} half-map $y_L$ satisfies the relationship
\[
\sgn\left(y_0+y_L(y_0) \right)=-\sgn(T_L) \quad \mbox{for} \quad y_0\in  I_L\setminus\{0\}.
\]
In addition,
when $0\in I_L$ and $y_L(0)\neq0$ or when $T_L=0$, then the relationship above also holds for $y_0=0$.
\item The backward Poincar\'{e} half-map $y_R$ satisfies the relationship
\[
\sgn\left(y_0+y_R(y_0) \right)=\sgn(T_R) \quad \mbox{for} \quad y_0\in  I_R\setminus\{0\}.
\]
In addition, 
when $0\in I_R$ and $y_R(0)\neq0$ or when $T_R=0$, then the relationship above also holds for $y_0=0$.
\end{enumerate}
\end{proposition} 

For the sake of simplicity, the next result is only given for the map $y_R$, which will be used later on in the proof of Theorem \ref{main}. A version for the map $y_L$ can be stated in an analogous way.

\begin{proposition}\label{prop:serietaylor}
Assume that  $0\in I_R$ and $y_R(0)=\widehat{y}_1<0$, then the  backward Poincar\'{e} half-map $y_R$ is a real analytic function in $I_R$ and its Taylor expansion around the origin writes as
\begin{equation*}
\label{eq:serietaylor}
y_R(y_0)=\widehat{y}_1+\frac{W_R\left(\hat{y}_1\right) y_0^2}{2 a_R^2 \hat{y}_1}+\mathcal{O}\left(y_0^3\right).
\end{equation*}
\end{proposition}

Again, for the sake of simplicity, the next result is only provided for the map $y_L$. An analogous result for the map $y_R$ can be also stated.

\begin{proposition}\label{prop:serieNewton-Puiseux}
Assume that  there exists a value  $\widehat y_0>0$ such that $y_L(\widehat y_0)=0$. Then, $a_L<0$, $\widehat y_0=\lambda_L$, that is, $\widehat y_0$ is the left endpoint of the definition interval $I_l$ of $y_L$, and 
the Poincar\'{e} half-map $y_L$ admits the Newton-Puiseux series expansion around the point $\widehat y_0$ given by
\[
\begin{aligned}
&y_L(y_0)={\displaystyle a_L\sqrt{\frac{2\lambda_L}{W_L(\lambda_L)}}\,(y_0-\lambda_L)^{1/2} +\mathcal{O}(y_0-\lambda_L).}\\
\end{aligned}
 \]
\end{proposition}

\subsection{Displacement function}\label{sec:disp}
Once the Poincar\'{e} half-maps have been characterized, a displacement function can be defined for system \eqref{cf}.

Suppose that  $I^b:=I_L\cap (I_R+b)\neq\emptyset$. The displacement function $\delta_b$ is, then, defined in $I^b$ as follows:
\begin{equation}\label{displacement}
\begin{array}{cccl}
\delta_b : & I^b & \longrightarrow & \mathbb{R} \\
  & y_0 & \longmapsto & \delta_b(y_0):= y_R^b(y_0)-y_L(y_0)=y_R(y_0-b)+b-y_L(y_0).
\end{array}
\end{equation}

From Propositions \ref{prop:defyL} and \ref{prop:defyR}, one has $I^b=[\lambda_b,\mu_b)$, where $\lambda_b=\max\{\lambda_L,\lambda_R+b\}$ and $\mu_b=\min\{\mu_L,\mu_R+b\}$. In addition, $\delta_b$ is continuous on $I^b$ and analytic  on $\Int(I^b)$.

\begin{remark}\label{delta-pa}
Notice that, by the continuity of $\delta_b$ on $I_b$ and the analyticity on $\Int(I_b)$, a crossing period annulus exists if, and only if,  $\delta_b(y_0)=0$ for every $y_0\in I^b$. Obviously, in this case, the $i$th order derivative satisfies $\delta_b^{(i)}(y_0)=0$ for every $y_0\in I^b$ and $i\in\mathbb{N}$. Of course, when $y_0=\lambda_b$, $\delta_b^{(i)}(y_0)=0$ refers to the lateral derivative.

\end{remark}

Now, some of the properties of $\delta_b$ (in particular, relevant expressions for the sign of the derivatives)  will be stated in the next proposition. Its proof can be seen in \cite{Caretalpre22}.

\begin{proposition}
\label{prop:derivdisplafunct} Let us consider the displacement function 
given in \eqref{displacement} for $b=0$. Suppose that $y_0^*\in \operatorname{int}(I^0)$ satisfies $\delta_0(y_0^*)=0$. Denote $y_1^*=y_R(y_0^*)=y_L(y_0^*)<0$ and define 
\begin{equation}
\label{eq:c0c1c2}
\begin{array}{l}
c_0:=a_Ra_L\left(a_RT_L-a_LT_R\right),\\
c_1:=a_RT_RD_L-a_LT_LD_R,\\
c_2:=a_L^2 D_R-a_R^2 D_L.\\
\end{array}
\end{equation}
 Then, the following statements hold:
 \begin{enumerate}[(a)]
\item The derivative of the displacement function $\delta_0$ defined in \eqref{displacement} verifies
\begin{equation}
\label{eq:signodeltaprima}
\sgn\left(\delta_0'(y_0^*)\right)=\sgn(F(y_0^*,y_1^*)),
\end{equation}
being 
\begin{equation}\label{CeF}
F(y_0,y_1)=c_0+c_1 y_0 y_1+c_2(y_0+y_1).
\end{equation}
\item \label{itembdisplafunct} Moreover, if $\delta_0'(y_0^*)=0$, 
then the second derivative of $\delta_0$ verifies
\begin{equation*}
\label{eq:signodeltasegunda}
\sgn\left(\delta_0''(y_0^*)\right)=\sgn\left(T_L \left( c_2 y_0^*+c_0\right) \right)=-\sgn\left(T_R \left( c_2 y_1^*+c_0\right) \right).
\end{equation*}
\end{enumerate}
\end{proposition}

\begin{remark}\label{remark:c0c1c2}
This remark is devoted to provide some useful and interesting relationships between the coefficients $\xi_0, \xi_\infty, c_0,c_1$ and $c_2$ (given in expressions  \eqref{ec:losxiybeta} and \eqref{eq:c0c1c2}), which will be used later on.

The set of polynomial functions $\{W_L,W_R\}$, with $W_L$ and $W_R$ defined in Propositions \ref{prop:defyL} and \ref{prop:defyR}, 
is linearly dependent if, and only if, $c_0=c_1=c_2=0$.

Moreover, the following equalities hold:
\begin{equation}
\label{ecu:relationc0c1c2}
c_0=a_Ra_L\xi_0,\qquad c_0\left(\begin{array}{c}
D_L \\D_R
\end{array}
\right)-c_2\left(\begin{array}{c}
-a_L T_L\\-a_RT_R
\end{array}
\right)+c_1\left(\begin{array}{c}
a_L^2 \\a_R^2
\end{array}
\right)=0,
\end{equation}

\begin{equation}
\label{c1xi0xiinf}
T_L c_1+a_L \xi_{\infty}=D_L T_R \xi_0,\quad\text{and}\quad T_R c_1+a_R \xi_{\infty}=D_R T_L \xi_0.
\end{equation}

\end{remark}

\section{Characterization of crossing period annuli}\label{sec:proof}

This section is dedicated to the proof of Theorem \ref{main}. It starts with a result on partial necessary conditions for the existence of a crossing period annulus. In particular, this result states that if system \eqref{s1} has a crossing period annulus, then it cannot have a sliding region.  This result has already been obtained in \cite{FREIRE2021124818}  by Freire et al. in the case that system \eqref{s1} is formed by two foci.

\begin{lemma}\label{lem:bceroTLTR}
If the piecewise linear differential system \eqref{s1} has a crossing period annulus, then  the condition (H) holds,  the value $\beta$ defined in \eqref{ec:losxiybeta} vanishes,  and  $\sgn(T_R)=-\sgn(T_L)$. \end{lemma}

\begin{proof}
Notice that, if system \eqref{s1} has a crossing period annulus,
 then, in particular, $a_{12}^L a_{12}^R>0$ and, therefore, system \eqref{s1} can be transformed into system \eqref{cf}, which will  also have a crossing period annulus. Hence, $I^b= [\lambda_b,\mu_b)\ne \emptyset$, the Poincar\'{e} half-maps are well defined and so, from Propositions \ref{prop:defyL} and \ref{prop:defyR}, we have that Hypothesis (H) holds.
 
 Now, we show that the existence of a crossing period annulus implies that $b=0$ and, consequently, $\beta=0$. Suppose, by reduction to absurdity, that system \eqref{cf} has a crossing period annulus and $b\ne0$. Let us assume that $b>0$, otherwise, by applying the transformation $(t,y) \mapsto (-t,-y)$, we can change the sign of $b$. This transformation also changes the signs of $T_L$ and $T_R$, but this will not be important in getting a contradiction. In the sequel, our reasoning distinguishes whether or not $b$ belongs to the interval $I^b$. 

On the one hand, let us consider $b\in I^b$. Then, $0\in I_R$ and $b\in I_L$. If $y_R(0)=0$, it follows that $\delta_b(b)=y_R(b-b)-y_L(b)+b=-y_L(b)+b>0$ and this contradicts the fact that $\delta_b(y_0)=0$ for every $y_0\in I^b$ (see Remark \ref{delta-pa}). If $y_R(0)<0$, then, from Proposition \ref{prop:serietaylor}, one obtains $y_R'(0)=0$. Thus, if $y_L(b)<0$, from Proposition \ref{prop:derivadasprimeraysegundayLyR}, one get  $y_L'(b)<0$; if, on the other hand, $y_L(b)=0$, then, from Proposition \ref{prop:serieNewton-Puiseux} (by taking $\widehat y_0=b$), one gets that 
\[
\lim_{y_0\searrow b}y_L'(y_0)=-\infty.
\]
In both cases, $\delta_b'(b)\ne 0$ which contradicts the fact that $\delta_b'(y_0)=0$ for every $y_0\in I^b$ (see Remark \ref{delta-pa}).

On the other hand, consider $b\notin I^b$. We know that $\lambda_b=\max\{\lambda_L,\lambda_R+b\}$. First, let us assume that $\lambda_b=\lambda_R+b$, which implies that $\lambda_L\leq \lambda_R+b $. Taking into account that $b\notin I^b$, we have that $\lambda_R>0$. Thus, by statement (b) of Proposition \ref{prop:defyR}, we have $y_R(\lambda_R)=0.$ Since $\lambda_R+b\in I_L$, then $y_L(\lambda_R+b)\leq 0$. Hence $\delta_b(\lambda_R+b)=y_R(\lambda_R)+b-y_L(\lambda_R+b)>0$, which contradicts the fact that $\delta_b(\lambda_R+b)=0$. Second, let us assume that  $\lambda_b=\lambda_L$, which implies that $\lambda_L\geq\lambda_R+b\geq b$. Taking into account that $b\notin I^b$, the last inequality implies, in fact, that $\lambda_L>b>0$.
Thus, by statement \eqref{itembyL} of Proposition \ref{prop:defyL}, we have $y_L(\lambda_L)=0$ and, then, by Proposition \ref{prop:serieNewton-Puiseux}, 
\begin{equation}\label{eq1}
\lim_{y_0\searrow \lambda_L}y_L'(y_0)=-\infty.
\end{equation}
From Remark \ref{delta-pa}, $y_L(\lambda_L)=0$ implies that $y_R(\lambda_L-b)=-b<0$ which, in turns, from statement \eqref{itembyL} of Proposition \ref{prop:defyR} and taking into account that $\lambda_L>b$, implies that $\lambda_L-b\in\Int(I_R)$. Hence, Proposition \ref{prop:derivadasprimeraysegundayLyR} implies that \begin{equation}\label{eq2}
y_R'(\lambda_L-b)<0.
\end{equation}
 The relationships \eqref{eq1} and \eqref{eq2} contradicts the fact that $\delta_b'(\lambda_L)=0$.
 
 Therefore, we have shown that the existence of a crossing period annulus implies that $b=0$ and, consequently, $\beta=0$.

Finally,  $b=0$ implies that $y_L(y_0)=y_R(y_0)$ for every $y_0\in I^b$ and, from Proposition \ref{rm:signoy0+y1}, it follows that $\sgn(T_R)=-\sgn(T_L)$ and the proof is finished.
 \end{proof}

\subsection{Proof of Theorem \ref{main}}

Let us start by assuming that the differential system \eqref{s1} has a crossing period annulus. From Lemma \ref{lem:bceroTLTR},  (H) holds, $\beta=0$, and $\sgn(T_L)=-\sgn(T_R)$. In addition, since $T_LT_R=0$ implies that $T_L=T_R=0$ and, therefore, $\xi_0=\xi_{\infty}=0$, then it only remains to show that $\xi_0=\xi_{\infty}=0$ for the case $T_L T_R<0$.

Recall that, under the first condition of (H), that is, $a_{12}^L a_{12}^R>0$,  system \eqref{s1} can be transformed into system \eqref{cf}, with $b=0$, which will also have a crossing period annulus.

From hypothesis  and taking into account Remark \ref{delta-pa}, the displacement function $\delta_b$ for $b=0$ verifies $\delta_0(y_0)=\delta_0'(y_0)=\delta_0''(y_0)=0$ for every $y_0\in \Int\big(I^0\big)$ and, by means of statement \eqref{itembdisplafunct} of Proposition \ref{prop:derivdisplafunct},
\[
T_L(c_2 y_0+c_0)=T_R(c_2 y_0+c_0)=0, \,\,\forall\,\, y_0\in \Int\big(I^0\big).
\]
Since $T_LT_R<0$, we have that $c_0=c_2=0$. In addition, from \eqref{eq:signodeltaprima} and \eqref{CeF}, we have that $c_1 y_0\, y_L(y_0)=c_1 y_0 \,y_R(y_0)=0$ for every $y_0\in \Int\big(I^0\big)$ and, consequently, $c_1=0$. 

From Remark \ref{remark:c0c1c2}, the relationship $c_0=c_1=c_2=0$ indicates that the polynomials $W_L(y)=D_Ly^2-a_LT_Ly+a_L^2$ and $W_R(y)=D_Ry^2-a_RT_Ry+a_R^2$  are linearly dependent. By one hand, if $a_L=0$, then $a_R=0$ and so $\xi_0=0$. Furthermore, in this case, from \eqref{ecu:yLyforaLaRnull}, 
one can see that the existence of a crossing period annulus, that is, the condition $y_L(y_0)=y_R(y_0)$ for $y_0\geq 0$ leads, by a direct computation, to $\xi_{\infty}=0$.
On the other hand, if $a_L\neq0$, then $a_R\neq0$. Thus, since $c_0=0$, from \eqref{ecu:relationc0c1c2}, one gets $\xi_0=0$ and so any of the relationships in  \eqref{c1xi0xiinf} implies $\xi_{\infty}=0$, because $c_1=0$.

Reciprocally, consider the planar piecewise linear differential system \eqref{s1} and assume that condition (H) holds, $\sgn(T_L)=-\sgn(T_R)$, and  $\beta=\xi_0=\xi_{\infty}=0$. Let us show the existence of a crossing period annulus for system \eqref{cf} and, consequently, for \eqref{s1}. 

Note that if $T_L T_R=0$, taking into account that $\sgn(T_L)=-\sgn(T_R)$, we have that $T_L=T_R=0$. Thus, from \eqref{integralF} and \eqref{integralB}, since the integrands are odd functions, it is trivial that $y_L(y_0)=y_R(y_0)=-y_0$ for every $y_0\in I^0$. This implies the existence of a crossing period annulus. Thus, for the rest of the proof, we can assume $T_LT_R<0$.

From \eqref{ecu:relationc0c1c2} and \eqref{c1xi0xiinf}, $\xi_0=\xi_{\infty}=0$ implies $c_0=c_1=0.$  Now, we show that $c_2=0$. Indeed, if  $a_L=a_R=0$, then $c_2=a_L^2 D_R-a_R^2 D_L=0$ and, otherwise, if $a_L^2+a_R^2\ne 0$, then the second relationship of \eqref{ecu:relationc0c1c2} implies $c_2=0$. 

Since $c_0=c_1=c_2=0$, from Remark \ref{remark:c0c1c2}, the polynomials $W_L(y)=D_Ly^2-a_LT_Ly+a_L^2$ and $W_R(y)=D_Ry^2-a_RT_Ry+a_R^2$  are linearly dependent, that is, $W_L=kW_R$. Moreover, $k>0$. Indeed, if $a_L^2+a_R^2\neq0,$ then $k>0$ immediately, otherwise, if $a_L=a_R=0$, from (H), we have $D_L,D_R>0$ and, again, $k>0$.

 Hence, $\sgn(a_L)=-\sgn(a_R)$, because $T_LT_R<0$. In addition, $\xi_{\infty}=0$ implies that $D_L=(T_L/T_R)^2D_R$. Thus, $k=(T_L/T_R)^2$ and 
\[
D_L=k D_R,\quad T_L=-\sqrt{k}\, T_R,\quad\text{and}\quad a_L=-\sqrt{k}\, a_R.
\]
Therefore,
\[
\operatorname{PV}\left\{\int_{y_L(y_0)}^{y_0}\dfrac{-y}{W_L(y)}dy\right\}=\frac1k \operatorname{PV}\left\{\int_{y_L(y_0)}^{y_0}\dfrac{-y}{W_R(y)}dy\right\}
\]
and the functions $q_L$ and $q_R$ defined in expressions \eqref{ecu:qL} and \eqref{ecu:qR} satisfy
\[
q_L(a_L,T_L,D_L)=q_L\left(-\sqrt{k}\, a_R,-\sqrt{k}\, T_R,k D_R\right)=\frac1k q_R(a_R,T_R,D_R).
\]
Now, from Propositions \ref{prop:defyL} and \ref{prop:defyR}, we see that $y_L$ and $y_R$ have the same integral characterization and, consequently, they coincide, that is, $y_L(y_0)=y_R(y_0)$ for $y_0\in I_L=I_R$. This implies the existence of a crossing period annulus and the proof is finished.

\section*{Acknowledgements}
VC is partially supported by the Ministerio de Ciencia, Innovaci\'on y Universidades, Plan Nacional I+D+I cofinanced with FEDER funds, in the frame of the project PGC2018-096265-B-I00. FFS is partially supported by the Ministerio de Econom\'{i}a y Competitividad, Plan Nacional I+D+I cofinanced with FEDER funds, in the frame of the project MTM2017-87915-C2-1-P. VC and FFS are partially supported by the Ministerio de Ciencia e Innovaci\'on, Plan Nacional I+D+I cofinanced with FEDER funds, in the frame of the project PID2021-123200NB-I00, the Consejer\'{i}a de Educaci\'{o}n y Ciencia de la Junta de Andaluc\'{i}a (TIC-0130, P12-FQM-1658) and by the Consejer\'{i}a de Econom\'{i}a, Conocimiento, Empresas y Universidad de la Junta de Andaluc\'{i}a (US-1380740, P20-01160). DDN is partially supported by S\~{a}o Paulo Research Foundation (FAPESP) grants 2022/09633-5, 2021/10606-0, 2019/10269-3, and 2018/13481-0, and by Conselho Nacional de Desenvolvimento Cient\'{i}fico e Tecnol\'{o}gico (CNPq) grants 438975/2018-9 and 309110/2021-1.

\section*{Declarations}
\subsection*{Ethical Approval} Not applicable
\subsection*{Competing interests} To the best of our knowledge, no conflict of interest, financial or other, exists.
\subsection*{Authors' contributions} All persons who meet authorship criteria are listed as authors, and all authors certify that they have participated sufficiently in the work to take public responsibility for the content, including participation in the conceptualization, methodology, formal analysis, investigation, writing-original draft preparation and writing-review \& editing.
\subsection*{Availability of data and materials} Data sharing not applicable to this article as no datasets were generated or analyzed during the current study.


\end{document}